\documentclass{amsart}
\usepackage{amssymb}
\usepackage{amsmath}
\usepackage{amsthm}
\usepackage{empheq}
\usepackage{enumerate}

\usepackage{float}
\restylefloat{table}

\usepackage[all]{xy}

\theoremstyle{definition}

\newtheorem*{remark}{Remark}

\newtheorem{theorem}{Theorem}[section]

\begin{document}

\title{Chudnovsky-Ramanujan Type Formulae for the Legendre Family}

\author{Imin Chen and Gleb Glebov}

\date{September 2017}



\address{Imin Chen \\
Department of Mathematics \\
Simon Fraser University \\
Burnaby \\
British Columbia \\
CANADA}

\email{ichen@sfu.ca}

\address{Gleb Glebov \\
Department of Mathematics \\
Simon Fraser University \\
Burnaby \\
British Columbia \\
CANADA}

\email{gglebov@sfu.ca}

\thanks{This work was supported by an NSERC Discovery Grant and SFU VPR Bridging Grant.}

\maketitle

\begin{abstract}
We apply the method established in our previous work to derive a Chudnovsky-Ramanujan type formula for the Legendre family of elliptic curves. As a result, we prove two identities for $1/\pi$ in terms of hypergeometric functions.
\end{abstract}

\tableofcontents

\section{Introduction}

In \cite{Ramanujan}, Ramanujan derived a number of rapidly converging series for $1/\pi$. Later, Chudnovsky and Chudnovsky \cite{Chudnovsky}, \cite{Chudnovsky2} derived an addtional such series based on the modular $j$-function, which is often used in practice for the computation of the digits of $\pi$. In \cite{Chen}, we generalized the method of \cite{Chudnovsky}, \cite{Chudnovsky2} to derive a complete list of Chudnovsky-Ramanujan type formulae for the modular $j$-function. The method is systematic, and in principle applicable to any family of elliptic curves parametrized by a triangular Fuchsian group of genus zero.

The purpose of this paper is to illustrate the method of \cite{Chen} in another case. Consider the Legendre family of elliptic curves given by
$$C_\lambda : y^2 = x(x - 1)(x - \lambda),$$
which is parametrized by the modular $\lambda$ function. The Picard-Fuchs differential equation for this family is well known and given by
\begin{equation}
\label{eq:Picard}
\lambda(1 - \lambda) \frac{d^2 P}{d\lambda^2} + (1 - 2\lambda) \frac{dP}{d\lambda} - \frac{P}{4} = 0,
\end{equation}
which is a hypergeometric differential equation with parameters $a = 1/2$, $b = 1/2$, $c = 1$, and it has three regular singular points: $0, 1, \infty$.

Let ${}_{2}F_{1}(a, b; c; z)$ denote the hypergeometric function. Kummer's method yields six distinct hypergeometric solutions of the form
$$\lambda^\alpha (1 - \lambda)^\beta {}_{2}F_{1}(a, b; c; \nu(\lambda)),$$
where $\nu(\lambda)$ is one of
$$\lambda, \quad 1 - \lambda, \quad \frac{1}{\lambda}, \quad \frac{1}{1 - \lambda}, \quad \frac{\lambda}{\lambda - 1}, \quad \frac{\lambda - 1}{\lambda}.$$
In particular, we have the solutions:

Around $\lambda = 0$:
\begin{align*}
{}_{2}F_{1}\bigg(\frac{1}{2}, \frac{1}{2}; 1; \lambda\bigg),& \\
(1 - \lambda)^{-1/2} {}_{2}F_{1}\bigg(\frac{1}{2}, \frac{1}{2}; 1; \frac{\lambda}{\lambda - 1}\bigg).&
\end{align*}

Around $\lambda = 1$:
\begin{align*}
{}_{2}F_{1}\bigg(\frac{1}{2}, \frac{1}{2}; 1; 1 - \lambda\bigg),& \\
\lambda^{-1/2} {}_{2}F_{1}\bigg(\frac{1}{2}, \frac{1}{2}; 1; \frac{\lambda - 1}{\lambda}\bigg).&
\end{align*}

Around $\lambda = \infty$:
\begin{align*}
\lambda^{-1/2} {}_{2}F_{1}\bigg(\frac{1}{2}, \frac{1}{2}; 1; \frac{1}{\lambda}\bigg),& \\
(1 - \lambda)^{-1/2} {}_{2}F_{1}\bigg(\frac{1}{2}, \frac{1}{2}; 1; \frac{1}{1 - \lambda}\bigg).&
\end{align*}

We can restrict out attention to only the first solution of each pair above because the second solution of each pair is obtained from the first by a Pfaff transformation.

Each of these solutions will be valid near one of the singular points $\lambda = 0, 1, \infty$ and will give rise to an expression of the period $P_1(\lambda)$ in terms of a hypergeometric function and $\lambda$.

Using the method of \cite{Chen}, we turn one of these period expressions into a Chudnovsky-Ramanujan type formulae valid near $\lambda = 0$. As a result, we derive the following identities for $1/\pi$ in terms of hypergeometric functions:
\begin{align}
\frac{8}{\pi} &= {}_{2}F_{1}\bigg(\frac{1}{2}, \frac{1}{2}; 1; \frac{1}{2}\bigg) {}_{2}F_{1}\bigg(\frac{3}{2}, \frac{3}{2}; 2; \frac{1}{2}\bigg), \label{eq:identity1} \\
\frac{1}{\pi} &= {}_{2}F_{1}\bigg(\frac{1}{2}, \frac{1}{2}; 1; -1\bigg)^2 - {}_{2}F_{1}\bigg(\frac{1}{2}, \frac{1}{2}; 1; -1\bigg) {}_{2}F_{1}\bigg(\frac{3}{2}, \frac{3}{2}; 2; -1\bigg). \label{eq:identity2}
\end{align}

\begin{remark}
Although the above identities could be verified by other means, the purpose of this work is to illustrate how systematic the method in \cite{Chen} is to discover and derive such formulae.
\end{remark}

\section{Periods and families of elliptic curves}

Let $C_\lambda : y^2 = x(x - 1)(x - \lambda)$ be the Legendre family. We first explain how to specify periods and quasi-periods for $C_\lambda$. Suppose
\begin{enumerate}[(1)]
\item $\alpha_1(\lambda)$ is a simple loop which encircles 0 and 1, and does not pass through $\lambda$;
\item $\alpha_2(\lambda)$ is a simple loop which encircles 1 and $\lambda$, and does not pass through 0.
\end{enumerate}
Then $\{\alpha_1(\lambda), \alpha_2(\lambda)\}$ forms a $\mathbb{Z}$-basis for $H_1(E_\lambda, \mathbb{Z})$ as depicted in \cite[Chapter VI, Section 1, Figure 6.5]{Silverman}, and it is possible to choose $\alpha_i(\lambda)$, $i = 1, 2$, so they vary continuously for $\lambda$ in a simply connected open subset of $\mathbb{P}^1(\mathbb{C}) \setminus \{0 ,1, \infty\}$. See \cite{Chen} for details.

We then define the periods $P_i$ and quasi-periods $Q_i$ for $C_\lambda$ as:
\begin{align*}
P_i &= \int_{\alpha_i} \frac{dx}{y}, \\
Q_i &= -\int_{\alpha_i} x\frac{dx}{y}.
\end{align*}
Up to a modification of the $\mathbb{Z}$-basis $\{\alpha_1(\lambda), \alpha_2(\lambda)\}$ for $H_1(C_\lambda, \mathbb{Z})$, that is, a transformation in $\text{SL}_2(\mathbb{Z})$, we obtain another $\mathbb{Z}$-basis $\{\gamma_1(\lambda), \gamma_2(\lambda)\}$ such that $\tau = P_2/P_1$ is in the fundamental domain for $\Gamma(2)$, and is near the points $0, 1, \infty$ in the sense that $\lambda(\tau)$, $1 - \lambda(\tau)$, $1/(1 - \lambda(\tau))$ have absolute value less than unity, respectively.

Under the change of variables
\begin{align*}
x &\mapsto x + \frac{\lambda + 1}{3}, \\
y &\mapsto \frac{y}{2},
\end{align*}
the Legendre family is put into the Weierstrass form:
\begin{align*}
E_\lambda : y^2 &= 4x^3 - \frac{4}{3} (\lambda^2 - \lambda + 1)x - \frac{4}{27} (\lambda + 1)(2\lambda - 1)(\lambda - 2) \\
&= 4(x - e_0)(x - e_1)(x - e_\lambda),
\end{align*}
where
$$e_0 = -\frac{1 + \lambda}{3}, \quad e_1 = \frac{2 - \lambda}{3}, \quad e_\lambda = \frac{2\lambda - 1}{3}.$$
This defines an isomorphism $\pi : E_\lambda \to C_\lambda$.

The normalized $j$-invariant $J = j/12^3$ of $C_\lambda$ is given by
$$J = \frac{4}{27} \frac{(\lambda^2 - \lambda + 1)^3}{\lambda^2 (1 - \lambda)^2}.$$
The map $\lambda \mapsto J(\lambda)$ gives the covering map $X(2) \to X(1)$, where $X(n)$ denotes the modular curve with full level $n$-structure.

The choice of $P_i$ and $Q_i$ for $C_\lambda$ above fixes the periods and quasi-periods of $E_\lambda$, which we denote by $\Omega_i$ and ${\rm H}_i$, respectively. In fact, we have that
\begin{align*}
P_i &= 2\Omega_i, \\
Q_i &= 2{\rm H}_i - \frac{1 + \lambda}{3} P_i.
\end{align*}

Fix an elliptic curve $E : y^2 = 4x^3 - g_2 x - g_3$. For $u \ne 0$, the map $\varphi_u : (x, y) \mapsto (u^2 x, u^3 y)$ gives an isomorphism $\varphi_u : E \xrightarrow{\cong} E'$ from an elliptic $E$ over $\mathbb{C}$ to the elliptic curve $E'$ over $\mathbb{C}$, where
$$E : y^2 = 4x^3 - g_2 x - g_3 \quad \text{and} \quad E' : y^2 = 4x^3 - g_2' x - g_3',$$
and
\begin{empheq}[left=\empheqlbrace]{align}
\label{eq:scale}
\begin{split}
g_2' &= u^4 g_2, \\
g_3' &= u^6 g_3, \\
\Delta(E') &= u^{12} \Delta(E).
\end{split}
\end{empheq}

By the uniformization theorem, $E(\mathbb{C}) = E_{\Lambda}(\mathbb{C})$ for some lattice $\Lambda \subset \mathbb{C}$. It then follows that $E'(\mathbb{C}) = E_{u^{-1} \Lambda}(\mathbb{C})$, and we have the following commutative diagram:
\begin{displaymath}
\xymatrix{
\mathbb{C} / \Lambda
\ar[d]_{z \mapsto u^{-1} z}
\ar[r]_{\iota_\Lambda}
& E_\Lambda(\mathbb{C}) = E(\mathbb{C})
\ar[d]^{\varphi_{u}} \\
\mathbb{C} / {u^{-1} \Lambda}
\ar[r]_{\iota_{u^{-1} \Lambda}}
& E_{u^{-1} \Lambda}(\mathbb{C}) = E'(\mathbb{C})
}
\end{displaymath}
so that the isomorphism $\varphi_u$ corresponds to scaling $\Lambda$ by $u^{-1}$.

We now introduce three families of elliptic curves: $E$, $E_\tau$, $E_\lambda$, and compare their discriminants, associated lattices, and periods.

Consider the elliptic curve $E$ over $\mathbb{C}$ given by
$$E : y^2 = 4x^3 - g_2 x - g_3, \quad \Delta(E) = \Delta = g_2^3 - 27g_3^2, \quad \Lambda(E) = \mathbb{Z} \omega_1 + \mathbb{Z} \omega_2.$$

Taking $u = \omega_1$, and using \eqref{eq:scale} we see that $E$ is isomorphic to
$$E_\tau : y^2 = 4x^3 - g_2(\tau) x - g_3(\tau), \quad \Delta(E_\tau) = \Delta(\tau) = g_2(\tau)^3 - 27g_3(\tau)^2, \quad \Lambda_\tau = \mathbb{Z} + \mathbb{Z} \tau,$$
where $\tau = \omega_2/\omega_1$.

To find a $u$ so that $E$ is isomorphic to $E_\lambda$, we claim that it is necessary to set
\begin{empheq}[left=\empheqlbrace]{align}
\nonumber
\begin{split}
u^4 g_2 &= \frac{4}{3} (\lambda^2 - \lambda + 1), \\
u^6 g_3 &= \frac{4}{27} (\lambda + 1)(2\lambda - 1)(\lambda - 2).
\end{split}
\end{empheq}
Dividing the second equation by the first, we obtain
$$u^2 \frac{g_3}{g_2} = \frac{(\lambda + 1)(2\lambda - 1)(\lambda - 2)}{9(\lambda^2 - \lambda + 1)},$$
or
$$u = \sqrt{\frac{(\lambda + 1)(2\lambda - 1)(\lambda - 2)}{9(\lambda^2 - \lambda + 1)} \frac{g_2}{g_3}}.$$
The above is a necessary condition for $u$. To show sufficiency, observe that
\begin{align*}
u^4 g_2 &= g_2 \frac{g_2^2}{g_3^2} \bigg\{\frac{(\lambda + 1)(2\lambda - 1)(\lambda - 2)}{9(\lambda^2 - \lambda + 1)}\bigg\}^2 \\
&= \frac{9g_2}{4} \frac{12}{J^{1/3} \Delta^{1/3}} \frac{J - 1}{J} \bigg\{\frac{(\lambda + 1)(2\lambda - 1)(\lambda - 2)}{9(\lambda^2 - \lambda + 1)}\bigg\}^2 \\
&= \frac{27J}{J - 1} \bigg\{\frac{(\lambda + 1)(2\lambda - 1)(\lambda - 2)}{9(\lambda^2 - \lambda + 1)}\bigg\}^2 \\
&= \frac{4}{3} (\lambda^2 - \lambda + 1),
\end{align*}
and
\begin{align*}
u^6 g_3 &= g_3 \frac{g_2^3}{g_3^3} \bigg\{\frac{(\lambda + 1)(2\lambda - 1)(\lambda - 2)}{9(\lambda^2 - \lambda + 1)}\bigg\}^3 \\
&= \frac{27g_3}{8} \frac{12^{3/2}}{\sqrt{J} \sqrt{\Delta}} \frac{J^{3/2}}{(J - 1)^{3/2}} \bigg\{\frac{(\lambda + 1)(2\lambda - 1)(\lambda - 2)}{9(\lambda^2 - \lambda + 1)}\bigg\}^3 \\
&= \frac{27}{8} \frac{J^{3/2}}{(J - 1)^{3/2}} \sqrt{\frac{J - 1}{27}} \bigg\{\frac{(\lambda + 1)(2\lambda - 1)(\lambda - 2)}{9(\lambda^2 - \lambda + 1)}\bigg\}^3 \\
&= \frac{12^{3/2} \sqrt{27}}{8} \frac{J}{J - 1} \bigg\{\frac{(\lambda + 1)(2\lambda - 1)(\lambda - 2)}{9(\lambda^2 - \lambda + 1)}\bigg\}^3 \\
&= \frac{12^{3/2} \sqrt{27}}{8} \frac{4}{729} (\lambda + 1)(2\lambda - 1)(\lambda - 2) \\
&= \frac{4}{27} (\lambda + 1)(2\lambda - 1)(\lambda - 2).
\end{align*}

Thus, taking
$$u = \sqrt{\frac{(\lambda + 1)(2\lambda - 1)(\lambda - 2)}{9(\lambda^2 - \lambda + 1)} \frac{g_2}{g_3}},$$
and using \eqref{eq:scale}, we see that $E$ is isomorphic to
\begin{align*}
E_\lambda : y^2 &= 4x^3 - \frac{4}{3} (\lambda^2 - \lambda + 1)x - \frac{4}{27} (\lambda + 1)(2\lambda - 1)(\lambda - 2), \\
\Delta(E_\lambda) &= 16\lambda^2 (1 - \lambda)^2, \quad \Lambda(E_\lambda) = \mathbb{Z} \Omega_1 + \mathbb{Z} \Omega_2.
\end{align*}

Taking
$$u = \sqrt{\frac{(\lambda + 1)(2\lambda - 1)(\lambda - 2)}{9(\lambda^2 - \lambda + 1)} \frac{g_2(\tau)}{g_3(\tau)}},$$
we see that $E_\tau$ is isomorphic to $E_\lambda$. Hence, we have that $\Lambda(E_\lambda) = \mu(\tau) \Lambda_\tau$, where
\begin{empheq}[left=\empheqlbrace]{align}
\label{eq:homothety}
\begin{split}
\mu(\tau) &= \sqrt{\frac{9(\lambda^2 - \lambda + 1)}{(\lambda + 1)(2\lambda - 1)(\lambda - 2)} \frac{g_3(\tau)}{g_2(\tau)}} \\
&= \frac{\Delta(\tau)^{1/12}}{J^{1/6}} \frac{(J - 1)^{1/4}}{27^{1/4}} \sqrt{\frac{9(\lambda^2 - \lambda + 1)}{(\lambda + 1)(2\lambda - 1)(\lambda - 2)}} \\
&= \frac{2^{1/3}}{27} \frac{\Delta(\tau)^{1/12}}{(\lambda (1 - \lambda))^{1/6}}.
\end{split}
\end{empheq}

\section{Limit values of the modular $\lambda$ function}

It is well known that for $\tau$ in the upper half-plane we have
$$\lambda(\tau) = 16q^{1/2} - 128q + 704q^{3/2} - O(q^2), \quad q = e^{2\pi i \tau},$$
or with $x = \sqrt{q}$,
$$\lambda(\tau) = 16x - 128x^2 + 704x^3 - O(x^4), \quad x = \sqrt{q}.$$
In what follows and throughout the paper, the notation $\tau \to \infty$ will be synonymous with ${\rm Im}(\tau) \to \infty$. Therefore, we have
$$\lim_{\tau \to \infty} \lambda(\tau) = \lim_{x \to 0} (16x - 128x^2 + 704x^3 - O(x^4)) = 0.$$
Moreover, in view of $\lambda(-1/\tau) = 1 - \lambda(\tau)$, we see that
$$\lim_{\tau \to 0} \lambda(\tau) = \lim_{\tau \to \infty} \lambda(-1/\tau) = \lim_{\tau \to \infty} (1 - \lambda(\tau)) = 1.$$
Furthermore, using $\lambda(\tau + 1) = \lambda(\tau)/(\lambda(\tau) - 1)$ together with $\lambda(-1/\tau) = 1 - \lambda(\tau)$, we find that
$$\lim_{\tau \to 1} \lambda(\tau) = \lim_{\tau \to \infty} \lambda(1 - 1/\tau) = \lim_{\tau \to \infty} \frac{\lambda(-1/\tau)}{\lambda(-1/\tau) - 1} = \lim_{\tau \to \infty} \frac{\lambda(\tau) - 1}{\lambda(\tau)} = \infty.$$
Below is the table summarizing these limit values.
\begin{table}[H]
\centering
\begin{tabular}{cc}
\hline
$\tau$ & $\lambda(\tau)$ \\
\hline
$\infty$ & $0$ \\
$0$ & $1$ \\
$1$ & $\infty$ \\
\hline
\end{tabular}
\caption{Limit values of $\lambda(\tau)$}
\end{table}

\section{Hypergeometric representations of $\Omega_1$}

Recall that the first solution of \eqref{eq:Picard} around $\lambda = 0$ is
$${}_{2}F_{1}\bigg(\frac{1}{2}, \frac{1}{2}; 1; \lambda\bigg).$$
Following the method in \cite{Archinard} with the homothety factor $\mu(\tau)$ in \eqref{eq:homothety}, we have that
$${}_{2}F_{1}\bigg(\frac{1}{2}, \frac{1}{2}; 1; \lambda\bigg) = \frac{(A + B\tau) \eta(\tau)^2}{(\lambda (1 - \lambda))^{1/6}}.$$
Replacing $\tau$ by $\tau + 2$ leaves $\lambda(\tau)$ invariant but changes $\eta(\tau)^2$ by $\zeta \eta(\tau)^2$, where $\zeta = e^{\pi i/3}$. Hence, we get the identity
$$(A + B\tau) \frac{\eta(\tau)^2}{(\lambda (1 - \lambda))^{1/6}} = (A + B(\tau + 2)) \frac{\zeta \eta(\tau)^2}{(\lambda (1 - \lambda))^{1/6}}$$
valid for $\tau$ around $\infty$. Since both sides are nonzero for $\tau$ around $\infty$, we obtain
$$A + B\tau = (A + B(\tau + 2)) \zeta$$
for $\tau$ around $\infty$. Upon solving for $B/A$ and taking the limit as $\tau \to \infty$, shows that $B = 0$. Recalling that
$$\lambda(\tau) = 16q^{1/2} - 128q + 704q^{3/2} - O(q^2), \quad q = e^{2\pi i \tau},$$
we find that $A = 16^{1/6} = 4^{1/3}$, and thus we have the identity
$${}_{2}F_{1}\bigg(\frac{1}{2}, \frac{1}{2}; 1; \lambda\bigg) = 4^{1/3} \frac{\eta(\tau)^2}{(\lambda (1 - \lambda))^{1/6}}.$$
We thus obtain the following period expression.

\begin{theorem}
\label{thm:period}
For $\tau$ around $\infty$, we have
$$\omega_1 = 2^{1/3} \pi (\lambda (1 - \lambda))^{1/6} \Delta(E)^{-1/12} {}_{2}F_{1}\bigg(\frac{1}{2}, \frac{1}{2}; 1; \lambda\bigg).$$
\end{theorem}

\begin{proof}
Since $\Delta(\tau) = (2\pi)^{12} \eta(\tau)^{24}$ and $\Delta(\tau) = \omega_1^{12} \Delta$, the result following from the identity above.
\end{proof}

We state the above period expression for $E$, but if $E = E_\lambda$, then the period expression becomes the classical relation:
\begin{equation}
\label{eq:classical}
P_1 = 2\pi {}_{2}F_{1}\bigg(\frac{1}{2}, \frac{1}{2}; 1; \lambda\bigg),
\end{equation}
where $P_1 = 2\Omega_1$, because $\Delta(E_\lambda) = 16\lambda^2 (1 - \lambda)^2$.

Now, we use the fact that the Galois covering $X(2) \to X(1)$ has Galois group $\text{GL}_2(\mathbb{F}_2)$, which is isomorphic to $S_3$. Explicitly, we have generators:
\begin{empheq}[left=\empheqlbrace]{align}
\label{eq:Galois}
\begin{split}
\lambda(\tau \pm 1) &= \frac{\lambda(\tau)}{\lambda(\tau) - 1}, \\
\lambda(-1/\tau) &= 1 - \lambda(\tau).
\end{split}
\end{empheq}
Applying these transformations to the above period expression around $\lambda = 0$ gives all period expressions around $\lambda = 0, 1, \infty$. We also require the following functional equations of the Dedekind eta function:
\begin{empheq}[left=\empheqlbrace]{align}
\label{eq:Dedekind}
\begin{split}
\eta(\tau \pm 1) &= e^{\pm \pi i/12} \eta(\tau), \\
\eta(-1/\tau) &= \eta(\tau) \sqrt{-i\tau}.
\end{split}
\end{empheq}

\begin{theorem}
\label{thm:transform}
For $\tau$ around 0, we have
$$\omega_1 = 2^{1/3} \frac{\pi i}{\tau} (\lambda (1 - \lambda))^{1/6} \Delta(E)^{-1/12} {}_{2}F_{1}\bigg(\frac{1}{2}, \frac{1}{2}; 1; 1 - \lambda\bigg).$$
\end{theorem}

\begin{proof}
This is proven by applying the transformation $\tau \to -1/\tau$ in Theorem \ref{thm:period} and using \eqref{eq:Galois} and \eqref{eq:Dedekind}.
\end{proof}

\begin{theorem}
For $\tau$ around 1, we have
$$\omega_1 = 2^{1/3} \frac{\pi i}{(\tau + 1) \sqrt{1 - \lambda}} (\lambda (1 - \lambda))^{1/6} \Delta(E)^{-1/12} {}_{2}F_{1}\bigg(\frac{1}{2}, \frac{1}{2}; 1; \frac{1}{1 - \lambda}\bigg).$$
\end{theorem}

\begin{proof}
This is proven by applying the transformation $\tau \to \tau + 1$ in Theorem \ref{thm:transform} and using \eqref{eq:Galois} and \eqref{eq:Dedekind}.
\end{proof}

\section{Complex multiplication and quasi-period relations}

Suppose $\tau$ satisfies $a\tau^2 + b\tau + c = 0$ for mutually coprime integers $a$, $b$, $c$ such that $a > 0$ and $-d = b^2 - 4ac$, that is,
\begin{equation}
\label{Chudnovsky}
\tau = \frac{-b + \sqrt{-d}}{2a}.
\end{equation}

Let $q = e^{2\pi i \tau}$ with ${\rm Im}(\tau) > 0$, and define
$$s_2(\tau) = \frac{E_4(\tau)}{E_6(\tau)} \bigg(E_2(\tau) - \frac{3}{\pi {\rm Im}(\tau)}\bigg),$$
where
\begin{align*}
E_2(\tau) &= 1 - 24 \sum_{n = 1}^\infty n \frac{q^n}{1 - q^n}, \\
E_4(\tau) &= 1 + 240 \sum_{n = 1}^\infty n^3 \frac{q^n}{1 - q^n}, \\
E_6(\tau) &= 1 - 504 \sum_{n = 1}^\infty n^5 \frac{q^n}{1 - q^n}.
\end{align*}
Then we have from \cite{Chen} that
\begin{equation}
\label{eq:Chudnovsky}
\Omega_1 {\rm H}_1 {\rm Im}(\tau) - \Omega_1^2 \left[{\rm Im}(\tau) \frac{3g_3}{2g_2} s_2(\tau)\right] = \pi
\end{equation}
by taking $E = E_\lambda$ so that $\omega_i = \Omega_i$ and $\eta_i = {\rm H}_i$.

\section{A Chudnovsky-Ramanujan type formula}

Using the differential relations derived by Bruns \cite[pp. 237-238]{Bruns} with
\begin{align*}
g_2 &= \frac{4}{3} (\lambda^2 - \lambda + 1), \\
g_3 &= \frac{4}{27} (\lambda + 1)(2\lambda - 1)(\lambda - 2),
\end{align*}
we obtain
\begin{align*}
\frac{d\Omega_1}{d\lambda} &= -\frac{{\rm H}_1}{2\lambda (\lambda - 1)} - \frac{2\lambda - 1}{6\lambda (\lambda - 1)}\Omega_1, \\
\frac{d{\rm H}_1}{d\lambda} &= \frac{\lambda^2 - \lambda + 1}{18\lambda (\lambda - 1)}\Omega_1 + \frac{2\lambda - 1}{6\lambda (\lambda - 1)}{\rm H}_1.
\end{align*}

\begin{remark}
These relations can also be obtained by using Fricke's and Klein's method of deriving the Picard-Fuchs differential \cite[pp. 33-34]{Modulfunctionen} if applied to $E_\lambda$.
\end{remark}

Using the first differential relation above, together with a period expression, yields a Chudnovsky-Ramanujan type formula.

\begin{theorem}
\label{thm:general}
If $\tau$ satisfies \eqref{Chudnovsky} and is around $\infty$, and
\begin{align*}
g_2 &= \frac{4}{3} (\lambda^2 - \lambda + 1), \\
g_3 &= \frac{4}{27} (\lambda + 1)(2\lambda - 1)(\lambda - 2),
\end{align*}
then we have that
$$-F^2 \bigg[\frac{2\lambda - 1}{3} + \frac{3g_3}{2g_2} s_2(\tau)\bigg] + \lambda (1 - \lambda) \frac{dF^2}{d\lambda} = \frac{2a}{\pi \sqrt{d}},$$
where $\lambda = \lambda(\tau)$ and $F = {}_{2}F_{1}(1/2, 1/2; 1; \lambda)$.
\end{theorem}

\begin{proof}
The first differential relation is equivalent to
\begin{equation}
\label{eq:Bruns}
{\rm H}_1 = -2\lambda (\lambda - 1) \frac{d\Omega_1}{d\lambda} - \frac{2\lambda - 1}{3}\Omega_1.
\end{equation}
Substitute \eqref{eq:Bruns} into \eqref{eq:Chudnovsky} to get
\begin{align*}
\frac{2\pi a}{\sqrt{d}} &= -\Omega_1 \bigg[2\lambda (\lambda - 1) \frac{d\Omega_1}{d\lambda} + \frac{2\lambda - 1}{3}\Omega_1 + \Omega_1 \frac{3g_3}{2g_2} s_2(\tau)\bigg] \\
&= -\Omega_1^2 \bigg[\frac{2\lambda - 1}{3} + \frac{3g_3}{2g_2} s_2(\tau)\bigg] + 2\lambda (1 - \lambda) \Omega_1 \frac{d\Omega_1}{d\lambda}.
\end{align*}
Then, substituting $\Omega_1 = \pi F$ from \eqref{eq:classical}, where $F = {}_{2}F_{1}(1/2, 1/2; 1; \lambda)$, into the above, we obtain
$$\frac{2a}{\pi \sqrt{d}} = -F^2 \bigg[\frac{2\lambda - 1}{3} + \frac{3g_3}{2g_2} s_2(\tau)\bigg] + \lambda (1 - \lambda) \frac{dF^2}{d\lambda}.$$
\end{proof}

Under the covering $X(2) \to X(1)$ given by
$$J = \frac{4}{27} \frac{(\lambda^2 - \lambda + 1)^3}{\lambda^2 (1 - \lambda)^2},$$
the only values for $\tau$ for which $\lambda(\tau) \in \mathbb{Q}$ correspond to $J = 1$ and $\lambda = 2, 1/2, -1$, each with ramification index 2. In the fundamental domain for $\Gamma(2)$, we have
\begin{align*}
\lambda\bigg(\frac{i \pm 1}{2}\bigg) &= 2, \\
\lambda(i) &= \frac{1}{2}, \\
\lambda(i + 1) &= -1.
\end{align*}
Using
$$\frac{3g_3}{2g_2} = \Omega_1^2 \frac{\pi^2}{3} \frac{E_6(\tau)}{E_4(\tau)}$$
in Theorem \ref{thm:general} yields
$$\frac{2a}{\pi \sqrt{d}} = -F^2 \bigg[\frac{2\lambda - 1}{3} + \frac{\pi^4}{3} \bigg\{E_2(\tau) - \frac{3}{\pi {\rm Im}(\tau)}\bigg\} F^2\bigg] + \lambda (1 - \lambda) \frac{dF^2}{d\lambda},$$
which can be evaluated at $\tau = i$. Recall that $E_2(\tau)$ is a quasi-modular form, and $E_2(-1/\tau) = \tau^2 E_2(\tau) - 6i \tau/\pi$ for $\tau$ in the upper half-plane. Setting $\tau = i$, shows that $E_2(i) = 3/\pi$, and we obtain \eqref{eq:identity1}. Furthermore, it is evident that $E_2(\tau + 1) = E_2(\tau)$ because $q(\tau + 1) = q(\tau)$. Thus, setting $\tau = i + 1$, we obtain \eqref{eq:identity2}.

\end{document}